\documentclass[11pt]{amsart}
\usepackage{amsmath}
\usepackage{amssymb}
\usepackage{enumerate}
\usepackage{euscript}
\usepackage{amscd, url}
\usepackage[all]{xy}

\newtheorem{thm}{Theorem}[section]
\newtheorem{lem}[thm]{Lemma}
\newtheorem{prop}[thm]{Proposition}

\theoremstyle{definition}

\numberwithin{equation}{section}
\theoremstyle{remark}

\newcommand{\qr}{r}

\newcommand{\bbc}{{\mathbb C}}

\newcommand{\bbq}{{\mathbb Q}}

\newcommand{\bbz}{{\mathbb Z}}
\newcommand{\calO}{{\mathcal O}}

\newcommand{\Disc}{{\operatorname{Disc}}}

\newcommand{\gl}{{\operatorname{GL}}}

\newcommand{\textmod}{{\text {\rm mod}}}

\newcommand{\Hom}{{\operatorname{Hom}}}

\newcommand{\Z}{\bbz}
\newcommand{\Q}{\bbq}

\newcommand{\C}{\bbc}

\title[$S_3$-sextic fields]
{An error estimate for counting $S_3$-sextic number fields}

\author[Takashi Taniguchi]
{Takashi Taniguchi}
\address{
Department of Mathematics,
Graduate School of Science, Kobe University,
1-1, Rokkodai, Nada-ku, Kobe 657-8501, Japan}
\address{
Department of Mathematics, Princeton University,
Fine Hall, Washington Road, Princeton, NJ 08540}
\email{tani@math.kobe-u.ac.jp}
\author[Frank Thorne]
{Frank Thorne}
\address{Department of Mathematics,
University of South Carolina,
1523 Greene Street, Columbia, SC 29208
}
\email{thorne@math.sc.edu}  
\date{\today}

\begin{document}
\maketitle

\begin{abstract}
In this note, we prove a power-saving remainder term for
the function counting $S_3$-sextic number fields.
We also give a prediction on the second main term.

In addition, we present numerical data on counting functions for $S_3$-sextic number fields. 
The data indicates
that our prediction is likely to be correct, and it also suggests the existence of additional lower order
terms which we have not yet been able to explain.
\end{abstract}

\section{Statement}\label{sec:statement}
We call a sextic number field $\widetilde K$ {\em $S_3$-sextic}
if $\widetilde K$ is Galois over $\Q$ with $\textnormal{Gal}(\widetilde K/\Q)$ isomorphic to the symmetric group $S_3$.
Let $N_6^\pm(X;S_3)$ be the number of
$S_3$-sextic fields $\widetilde K$ with $0<\pm\Disc(\widetilde K)<X$.
The primary term of $N_6^\pm(X;S_3)$ was obtained in independent works of
Belabas-Fouvry \cite{befo} and Bhargava-Wood \cite{bhwo},
and in this article we prove the following power-saving remainder term.
\begin{thm}\label{thm:maintheorem}
We have
\begin{equation}\label{eq:weak}
N_6^\pm(X;S_3)=\frac{C^\pm}{12}\prod_p c_p\cdot X^{1/3}+O(X^{1/3-5/447+\epsilon}),
\end{equation}
where $C^+=1,C^-=3$, the product is over all primes, and
\begin{equation*}
c_p=\begin{cases}
(1-p^{-1})(1+p^{-1}+p^{-4/3}) & p\neq3,\\
(1-\frac13)(\frac43+\frac1{3^{5/3}}+\frac2{3^{7/3}})
& p=3.
\end{cases}
\end{equation*}
\end{thm}
Moreover, under a natural but rather strong {\em assumption}
of uniformity estimates for certain counting functions of cubic fields,
we obtain
\begin{equation}\label{eq:strong}
N_6^\pm(X;S_3)
=\frac{C^\pm}{12}\prod_p c_p\cdot X^{1/3}
+\frac{4K^\pm\zeta(1/3)}{5\Gamma(2/3)^3}\prod_p k_p\cdot X^{5/18}
+o(X^{5/18}),
\end{equation}
where $K^+=1, K^-=\sqrt3$ and
\begin{equation*}
k_p=\begin{cases}
1+\frac1{p^{13/9}(1+p^{-1})}
\left(1-\frac1{p^{2/9}}-\frac1{p^{5/9}}-\frac1{p^{2/3}}\right) & p\neq3,\\
\frac14\left(\frac{11}3-\frac1{3^{2/3}}+\frac1{3^{8/9}}
	+\frac2{3^{13/9}}-\frac1{3^{14/9}}-\frac2{3^{19/9}}\right) & p=3.
\end{cases}
\end{equation*}
As in \cite{befo} and \cite{bhwo}, we relate counting $S_3$-sextic fields
to counting non-cyclic cubic fields with certain local completions.
These cubic fields may then be counted using our
previous work \cite{scc}. We may obtain a power saving error term simply by
quoting our previous results, but we improve on this
by applying the methods used in the proofs in \cite{scc}. This amounts to computing the Fourier transform
of a function related to these local completions, and this computation was essentially carried out in \cite{lbc}. 

We also present numerical data for $N_6^{\pm}(X; S_3)$ for $X \leq 10^{23}$, computed by Cohen and the second author in \cite{CT3},
and verified independently by the present authors for 
$X \leq 5 \cdot 10^{18}$ by a second method. Interestingly, our computations
suggest that \eqref{eq:strong} is likely to be correct, but with additional lower order terms
which we were not able to explain. Our data also suggests the existence of 
surprising biases in arithmetic progressions, for example modulo $5$, which cannot be explained
by any heuristic of which we are aware.

\section*{Acknowledgments}
We are grateful to Karim Belabas, Simon Rubinstein-Salzedo, and an anonymous referee for useful comments.

Thorne's work is partially supported by the National Science Foundation under Grant No. DMS-1201330.

\section{Proof}
For a non-cyclic cubic field $K$, let $\widetilde K$ denote its Galois closure.
Then the map $K\mapsto\widetilde K$ gives a canonical bijection between the set of
isomorphism classes of non-cyclic cubic fields and the set of
isomorphism classes of $S_3$-sextic fields.
Let us compare $\Disc(K)$ and $\Disc(\widetilde K)$.
They have the same sign, and if we write
\[
\Disc(K)=\pm\prod p^{e_p(K)},
\qquad
\Disc(\widetilde K)=\pm\prod p^{e_p(\widetilde K)},
\]
we have the following.
\begin{lem}\label{lem:relation_disc}
\begin{enumerate}[{\rm (1)}]
\item
If $K$ is not totally ramified at $p$, then $e_p(\widetilde K)=3e_p(K)$.
\item
If $K$ is totally ramified at $p$ and $p\neq3$, then
$e_p(\widetilde K)=2e_p(K)=4$.
\item
If $K$ is totally ramified at $p=3$,
then $e_p(\widetilde K)=7,8$ or $11$ according as $e_p(K)=3,4$ or $5$.
\end{enumerate}
\end{lem}
\begin{proof}
Equivalent statements appear in \cite{befo} and \cite{bhwo}, and we give a proof for the convenience of the reader.
Let $F = \Q(\sqrt{\Disc(K)})$ be the quadratic resolvent field of $K$ (equivalently, the unique
quadratic subfield of $\widetilde{K}$).
We have the classical formula (see, e.g. Theorem 2.5.1 and Lemma 10.1.27 of \cite{coh})
\begin{equation}
\Disc(\widetilde{K}) = \Disc(K)^2 \Disc(F).
\end{equation}
Therefore, for $p > 2$, $e_p(\widetilde{K})$ is equal to $2 e_p(K) + a$, where $a = 0$ or $1$
depending on whether $e_p(K)$ is even or odd. 

For $p = 2$, observe that $2$ can ramify in $F$ only if it ramifies in $K$. If $(2) = \mathfrak{p}_1^2 \mathfrak{p}_2$
in $K$, then in ${\widetilde{K}}$, $(2)$ must split into three ideals with ramification index 2. Therefore $2$ must
ramify in $F$ with $e_2(F) = e_2(K)$ so that $e_2(\widetilde{K}) = 3 e_2(K)$. 
If $(2) = \mathfrak{p}^3$ in $K$, then $2$ is tamely ramified in $K$, and therefore $\widetilde{K}$,
so that $(2)$ splits into two ideals of ramification index 3 in $\widetilde{K}$. This implies that $(2)$ is unramified in
$F$, so that $e_2(\widetilde{K}) = 2 e_2(K)$.

\end{proof}

In particular,
$e_p(K)$ determines $e_p(\widetilde K)$ uniquely
except for the case $p=2$ and $e_2(K)=2$, 
while in this case $e_2(\widetilde K)$ is either $6$ or $4$
according as $K$ is partially or totally ramified at $2$.

Let us briefly explain our approach. If we ignore the
ramification over the prime $3$, then Lemma \ref{lem:relation_disc}
implies
\begin{equation}\label{eq:relation_disc_fake}
\Disc(\widetilde K) =^* \qr^{-2}\Disc(K)^3,
\end{equation}
where $\qr$ is the product
of all primes where $K$ is totally ramified.\footnote{We have starred two equalities which are not actually true as stated.
We correct them in \eqref{eq:relation_disc} and \eqref{eq:relation_count} respectively.}
Denoting by $N_3^\pm(X;\qr)$ the number
of non-cyclic cubic fields $K$ with $\qr$ as above
and $0<\pm\Disc(K)<X$, then
\begin{equation}\label{eq:relation_count_fake}
N_6^\pm(X;S_3) =^* \sum_\qr N_3^\pm(\qr^{2/3}X^{1/3};\qr).
\end{equation}
Here the sum is over all square-free integers $\qr$.
However, \eqref{eq:relation_count_fake} may not be true
because of the ramification at $3$,
so we specify the completion $A$ of $K$ at $3$ and
count for each $A$.

Let $A$ denote an \'{e}tale cubic algebra over $\Q_3$
(i.e., a direct product of field extensions of $\Q_3$ whose degrees add to 3)
and $\qr$ a square-free integer coprime to $3$.
Let $\mathcal K_3(A,\qr)$ be the set of non-cyclic cubic fields $K$
satisfying (i) $K\otimes_\Q\Q_3\cong A$,
(ii) $K$ is totally ramified at all prime divisors of $\qr$, and
(iii) $K$ is not totally ramified at any prime $p \nmid 3r$.
Let $\widetilde A$ be the sextic
algebra over $\Q_3$ isomorphic to $\widetilde K\otimes_\Q\Q_3$ for 
$K\in\mathcal K_3(A,\qr)$, which does not depend on $K$.
Let $\Disc_3(A)$ and $\Disc_3(\widetilde A)$ be the $3$-parts of their
discriminants; e.g., write $\Disc(A) = u \Disc_3(A)$, where $u$ is a $3$-adic unit\footnote{Observe that $\Disc(A)$ and $u$
are only defined up to squares of $3$-adic units, but $\Disc_3(A)$ is well defined.}, and similarly
for $\Disc(\widetilde A)$.
Then for $K\in\mathcal K_3(A,\qr)$,
instead of \eqref{eq:relation_disc_fake} we have
\begin{equation}\label{eq:relation_disc}
\Disc(\widetilde K)=\qr^{-2} m_A^{-1} {\Disc_3(A)^3}\Disc(K)^3,
\end{equation}
with
$m_A:=\Disc_3(A)^3/\Disc_3(\widetilde A)$.
Let
$N_3^\pm(X;A,\qr)$
denote the number of $K\in\mathcal K_3(A,\qr)$
with $0<\pm\Disc(K)<X$.
We will use a formula of the form
\begin{equation}\label{eq:cubic_density}
\begin{split}
N_3^\pm(X;A,\qr)
&=\eta_3(A)\eta(\qr)\prod_p(1-p^{-2})\frac{C^\pm}{12}X\\
&+\theta_3(A)\theta(\qr)\prod_p\left(1-\frac{p^{1/3}+1}{p(p+1)}\right)\frac{4K^\pm\zeta(1/3)}{5\Gamma(2/3)^3}X^{5/6}+O(\qr^\alpha X^\beta).
\end{split}
\end{equation}
Here $\eta_3(A)$ and $\theta_3(A)$ are ``local densities'' of $A$
computed in \cite{scc},
$\eta$ and $\theta$ are multiplicative functions satisfying
\[
\eta(p)=\frac{1}{p^2(1+p^{-1})},
\qquad
\theta(p)=\frac{1}{p^2(1+p^{-2/3}+p^{-1}+p^{-4/3})}
\]
for any prime $p$, and $\alpha,\beta$ are certain real constants.
By Theorem 1.2 in \cite{scc}, \eqref{eq:cubic_density} is true
with $\alpha=40/23, \beta=18/23+\epsilon$ and this suffices
to obtain \eqref{eq:weak} with a larger error term of
$O(X^{1/3-5/744+\epsilon})$. In this paper, we improve the
estimate as follows.
\begin{thm}\label{thm:exponent}
The formula \eqref{eq:cubic_density}
is true for $\alpha=7/23+\epsilon$, $\beta=18/23+\epsilon$.
\end{thm}
We postpone its proof to the next section, and continue the
proof of \eqref{eq:weak} and \eqref{eq:strong}.
Let $N_6^\pm(X;A)$ be the number of $S_3$-sextic fields $\widetilde K$
such that $K\otimes_\Q\Q_3\cong A$.
Then by \eqref{eq:relation_disc},
\begin{equation}\label{eq:relation_count}
N_6^\pm(X;A)=\sum_{3\nmid \qr}N_3^\pm(\qr^{2/3}m_A^{1/3}X^{1/3};A,\qr)
\end{equation}
where the sum is over all square-free integers coprime to $3$. Therefore, our results follow
from \eqref{eq:relation_count}, \eqref{eq:cubic_density}, and a computation, the details of which follow.

We choose $Q$ and split this sum into $\qr<Q$ and $\qr\geq Q$.
By \cite[Lemma 3.4]{scc} we have the estimate
$N_3^\pm(X;A,\qr)=O(\qr^{-2+\epsilon}X)$.
Hence the latter sum is bounded by $O(Q^{-1/3+\epsilon}X^{1/3})$.
On the other hand it is easy to see that
\begin{gather*}
\begin{split}
\sum_{3\nmid \qr,\qr<Q}\eta(\qr)\qr^{2/3}
&=\prod_{p\neq3}(1+\eta(p)p^{2/3})+O({Q^{-1/3+\epsilon}}),\\
\sum_{3\nmid \qr,\qr<Q}\theta(\qr)\qr^{5/9}
&=\prod_{p\neq3}(1+\theta(p)p^{5/9})+O({Q^{-4/9+\epsilon}}).
\end{split}
\end{gather*}
We define
\begin{gather*}
c_p:=(1+\eta(p)p^{2/3})(1-p^{-2}),
\quad
k_p:=(1+\theta(p)p^{5/9})
\left(1-\frac{p^{1/3}+1}{p(p+1)}\right)
\qquad
(p\neq3),
\end{gather*}
which coincide with the constants given in Section \ref{sec:statement}.
We also put
\[
\eta_3'(A):=(1-3^{-2})\eta_3(A)m_A^{1/3},
\qquad
\theta_3'(A):=\left(1-\frac{3^{1/3}+1}{3(3+1)}\right)\theta_3(A)m_A^{5/18}.
\]
Then by \eqref{eq:relation_count} and \eqref{eq:cubic_density}, and ignoring a negligible
$O(Q^{-4/9 + \epsilon} X^{5/18})$ term, 
we have
\begin{equation}\label{eq:count_A}
\begin{split}
N_6^\pm(X;A)
=&\eta_3'(A)\prod_{p\neq3}c_p\cdot\frac{C^\pm}{12}X^{1/3}
+\theta_3'(A)\prod_{p\neq3}k_p\cdot\frac{4K^\pm\zeta(1/3)}{5\Gamma(2/3)^3}X^{5/18}\\
&+O(X^{\beta/3}\sum_{\qr<Q}\qr^{\alpha+2\beta/3})+O(Q^{-1/3+\epsilon}X^{1/3}).
\end{split}
\end{equation}
The first $O$-term is
$O(Q^{\alpha+2\beta/3+1}X^{\beta/3})$, and we choose
$Q=X^{\frac{1-\beta}{3\alpha+2\beta+4}}$ to obtain
an error
of
$O(X^{\frac13(1-\frac{1-\beta}{3\alpha+2\beta+4})+\epsilon})$
in \eqref{eq:count_A}. With our constants
$\alpha=7/23+\epsilon$ and $\beta=18/23+\epsilon$,
this is $O(X^{\frac13(1-\frac{5}{149})+\epsilon})$.
If \eqref{eq:cubic_density} is true for e.g., $\alpha=-1,\beta=1/2$,
this is $O(X^{1/4+\epsilon})$ and we would obtain the second main term.
Such an estimate might be true,
but it seems difficult to prove; moreover, our numerical data (see Section \ref{sec_data})
suggests that perhaps such a strong estimate is {\itshape not} true.

Recall that
\[
N_6^\pm(X;S_3)=\sum_A N_6^\pm(X;A)
\]
where $A$ in the right hand side
runs through all the \'{e}tale cubic algebras over $\Q_3$. (There are finitely many,
as there are finitely many field extensions of $\Q_3$ of degree $\leq 3$.)
Hence the contribution to the main term of $N_6^\pm(X;S_3)$
from the prime $3$ is given by
\[
c_3:=\sum_A\eta_3'(A)=(1-3^{-2})\sum_A\eta_3(A)m_A^{1/3}.
\]
The local density $\eta_3(A)$ is given in the tables in
Section 6.2 of \cite{scc}. Also, $m_A$ is equal to $1$, $9$, or $81$
depending on whether the $3$-adic valuation of $\Disc(A)$ is less than $3$,
equal to $3$, or greater than $3$. We therefore compute that
\begin{equation}\label{c3_comp}
c_3=(1-3^{-2}) \cdot \frac{1+ \frac{1}{3}+ \frac2{27} \cdot 3^{2/3}+\frac{1}{27} \cdot 3^{4/3}}{1 + \frac{1}{3}},
\end{equation}
which is equal to the quantity given in Section \ref{sec:statement}.
Similarly, the contribution to the secondary term is given by
\begin{equation}\label{k3_comp}
\sum_A\theta_3'(A)
=\left(1-\frac{3^{1/3}+1}{3(3+1)}\right)\sum_A\theta_3(A)m_A^{5/18},
\end{equation}
and a similar calculation yields the value of
$k_3$ given in Section \ref{sec:statement}.

%

\section{Proof of Theorem \ref{thm:exponent}}
In this section, we prove Theorem \ref{thm:exponent}
by following the arguments of \cite{lbc} and \cite{scc}. 

A brief sketch of our proof is as
follows.
In \cite{scc}, we counted cubic fields in terms of contour integrals of certain zeta functions
introduced by Shintani \cite{shintani}, associated to the space of {\itshape binary cubic forms}.
Our method is naturally compatible with
``local specifications'' such as those appearing in \eqref{eq:cubic_density}, and the error terms
of \eqref{eq:cubic_density} depend on the ``shape'' of these local specifications -- 
in particular, 
on the Fourier transforms of their indicator functions.
We establish fairly sharp 
bounds for these Fourier transforms on average, which lead to reasonably good bounds
on the error terms in \eqref{eq:cubic_density} (in $\alpha$-aspect) and therefore in Theorem 
\ref{thm:maintheorem}.

We follow the notations of \cite{lbc} and \cite{scc},
but recall the most basic ones.
Let
\begin{equation}
V(\Z):=\{x=(x_1,x_2,x_3,x_4)=x_1u^3+x_2u^2v+x_3uv^2+x_4v^3\mid x_1, x_2, x_3, x_4\in\Z\}
\end{equation}
be the lattice of integral binary cubic forms, with its
usual action of ${\rm GL}_2(\Z)$.
Then there is a discriminant preserving
bijection between ${\rm GL}_2(\Z)\backslash V(\Z)$
and the set of isomorphism classes of cubic rings.\footnote{A {\itshape cubic}
ring is a commutative ring which is free of rank 3 as a $\Z$-module.}

We now define these indicator functions.
Let $p\neq 3$ be a prime. 
We define $\Phi_p: V(\Z) \rightarrow \C$ to be the characteristic function of those $x \in V_{\Z}$
whose corresponding cubic ring is either nonmaximal at $p$, or maximal and totally ramified at $p$.
We similarly define $\Psi_p$ by requiring the cubic ring to be both maximal and totally ramified
at $p$. These two functions factor through
the reduction map $V(\Z)\rightarrow V(\Z/p^2\Z)$,
and we also write $\Phi_p$, $\Psi_p$ for these functions on $V(\Z/p^2\Z)$.

The prime $3$ demands a special treatment.
We fix an \'etale cubic algebra $A$ over $\Q_3$
throughout this section; note that since there are only finitely many $A$,
uniformity in our error terms with respect to $A$ is automatic.
Let $\Phi_A$ be the characteristic function on $V(\Z)$ or $V(\Z/27\Z)$
corresponding to cubic rings $R$
such that $R\otimes\Z_3\cong\mathcal O_A$,
where $\mathcal O_A$ is the integral closure\footnote
{If $A$ is of the form $A=\prod A_i$ where $A_i/\Q_3$ are field extensions,
then $\mathcal O_A=\prod {\mathcal O_{A_i}}$
where each $\mathcal O_{A_i}$ is the integer ring of $A_i$.}
of $\Z_3$ in $A$.
This $\Phi_A$ factors through $V(\Z)\rightarrow V(\Z/27\Z)$.

Let $r$ and $q$ be squarefree integers satisfying $(q,r)=(qr,3)=1$.
We put $\Phi_q=\prod_{p\mid q}\Phi_p$ and $\Psi_r=\prod_{p\mid r}\Psi_p$, and
define the zeta functions
\begin{equation}
\xi_{r,q}^\pm(s)
:=\sum_{\substack{x\in\gl_2(\Z)\backslash V(\Z) \\ \pm\Disc(x) > 0}}
\Phi_A(x)\Psi_r(x)\Phi_q(x)\frac{|{\rm Stab}(x)|^{-1}}{|\Disc(x)|^s}.
\end{equation}
As in \cite{scc}, Theorem \ref{thm:exponent}
follows from uniform estimates for the zeta functions
$\widehat{\xi_{r,q}}{}^{\pm}(s)$ which are dual to $\xi_{r,q}^\pm(s)$.

Let $V^\ast = \Hom(V(\Z), \Z)$ be the dual space of $V$.
The (finite) Fourier transform of $\Psi_r$, a function
of $y\in V^\ast(\Z/r^2\Z)$, is defined by
\begin{equation}
\widehat{\Psi_r}(y):=\frac{1}{r^8}\sum_{x\in V(\Z/r^2\Z)}
\Psi_r(x)\exp\left(2\pi\sqrt{-1}\cdot\frac{[x,y]}{r^2}\right),
\qquad
y\in V^\ast(\Z/r^2\Z),
\end{equation}
where
$[x, y] = x_1 y_1 + x_2 y_2 + x_3 y_3 + x_4 y_4
$
(using the coordinate system on $V^*(\Z)$ induced by the canonical pairing),
and we define $\widehat{\Phi_A}$ and $\widehat{\Phi_q}$ similarly.
Then the dual zeta function is defined by
\begin{equation}
\widehat{\xi_{r,q}}{}^{\pm}(s)
:=\sum_{\substack{y\in\gl_2(\Z)\backslash V^\ast(\Z)}}
\widehat{\Phi_A}(y)\widehat{\Psi_r}(y)\widehat{\Phi_q}(y)
\frac{|{\rm Stab}(y)|^{-1}}{|P^\ast(y)/3^{12}r^8q^8|^s}.
\end{equation}
The `dual discriminant' $P^{\ast}(y)$ is the same as the `ordinary'
discriminant $P(x)$ or $\Disc(x)$ on $V(\Z)$ apart from some $3$-adic factors; 
we refer to Section 2 of \cite{lbc}
for details. 

Because of the functional equation
\begin{equation}
\begin{pmatrix}
\xi_{r,q}^+(1-s)\\
\xi_{r,q}^-(1-s)\\
\end{pmatrix}
=
\frac{3^{6s-2}}{2\pi^{4s}}
\Gamma(s)^2\Gamma\left(s-\tfrac16\right)\Gamma\left(s+\tfrac16\right)
\begin{pmatrix}
\sin 2\pi s&\sin \pi s\\
3\sin \pi s&\sin2\pi s\\
\end{pmatrix}
\begin{pmatrix}
\widehat{\xi_{r,q}}{}^+(s)\\
\widehat{\xi_{r,q}}{}^-(s)\\
\end{pmatrix},
\end{equation}
the estimate of the $O$-term in \eqref{eq:cubic_density}
is reduced to estimates for these dual zeta functions
$\widehat{\xi_{r,q}}{}^{\pm}(s)$ which are uniform
with respect to $r$ and $q$.
We write
\begin{equation}
\widehat{\xi_{r,q}}^\pm(s)
:=\sum_{\mu_n}\frac{c_{r,q}^\pm(\mu_n)}{\mu_n^s},
\end{equation}
the sum being over $\mu_n\in\frac{1}{3^{12}r^8q^8}\Z$.
We fix a choice of sign and drop $\pm$ from our notation.
The following bound essentially
follows from Theorem 4.1 in \cite{scc}.
\begin{prop}
For any fixed $\epsilon>0$,
we have the bounds
\begin{align}
\sum_{\mu_n<X}|c_{r,q}(\mu_n)|&\ll
(rq)^{2+\epsilon}X,\label{eq:coeff_bound_1}\\
\sum_{\mu_n<X}|c_{r,q}(\mu_n)|&\ll
(rq)^{1+\epsilon}X+(rq)^{-1+\epsilon},\label{eq:coeff_bound_2}
\end{align}
uniformly for all $r,q$ and $X$.
\end{prop}
\begin{proof}
In \cite{lbc}, we gave explicit formulas for the Fourier
transforms $(\Phi_p-\Psi_p)^\wedge=\widehat{\Phi_p}-\widehat{\Psi_p}$
and $\widehat{\Phi_p}$ in Theorems 6.3 and 6.4, respectively, so
a formula for $\widehat{\Psi_p}$ follows by linearity.
We introduce a function $\Phi^\ast_p$ on $V^\ast(\Z/p^2\Z)$ by
\begin{equation}\label{eq:Phi^ast}
\Phi^\ast_p(b)=
	\begin{cases}
		p^{-5}	& \text{if $b$ is of type $(1^3_{\rm max})$},\\
		|\widehat{\Phi_p}(b)| & \text{otherwise}.
	\end{cases}
\end{equation}
Then we have $|\widehat{\Phi_p}|\leq \Phi^\ast_p$ and
$|\widehat{\Psi_p}|\leq (1+p^{-2})\Phi^\ast_p$.
Let $c=\prod_p(1+p^{-2})$.
Note the trivial bound $|\widehat{\Phi_A}|\leq1$.
Therefore $\widehat{\xi_{r,q}}{}^\pm(s)$ is bounded coefficientwise by
\begin{equation}\label{eq:bounded-coefficientwise}
\sum_{\substack{y\in\gl_2(\Z)\backslash V^\ast(\Z)}} c
\Phi^\ast_{rq}(y)\frac{|{\rm Stab}(y)|^{-1}}{|P^\ast(y)/3^{12}r^8q^8|^s}.
\end{equation}
Here $\Phi^\ast_{rq}=\prod_{p\mid rq}\Phi^\ast_p$.
If $\Phi^\ast_{rq}(y)$ in above were replaced with $|\widehat{\Phi_{rq}}(y)|$,
then \eqref{eq:bounded-coefficientwise} is,
in the notation of Section 4 in \cite{scc}, given by
\begin{equation}\label{eq:bounded-vertual}
c\sum_{\mu_n}\frac{b_{rq}(\mu_n)}{(\mu_n/3^{12})^s}.
\end{equation}
So the bounds of this proposition
follow from Theorem 4.1 in \cite{scc}.
Our actual \eqref{eq:bounded-coefficientwise} is slightly different
from \eqref{eq:bounded-vertual} because of \eqref{eq:Phi^ast}, but
we can nevertheless easily modify the proof of Theorem 4.1 in \cite{scc}
for our case and obtain the same estimate. We omit the detail.
\end{proof}
Similarly to Proposition 4.2 in \cite{scc},
we have the following corollary.
\begin{prop}\label{prop:Dirichlet_bound}
Let $z\geq r^{-2}q^{-2}$.
For a fixed $0<\delta<1$ (and $\epsilon>0$), we have the bounds
\begin{align}
\sum_{\mu_n<z}|c_{r,q}(\mu_n)|/\mu_n^\delta
\ll (rq)^{3\delta-1+\epsilon}+(rq)^{1+\epsilon}z^{1-\delta}.
\end{align}
We also have, for any fixed $\delta>1$,
\begin{equation}
\sum_{\mu_n>z}|c_{r,q}(\mu_n)|/\mu_n^\delta
\ll(rq)^{1+\epsilon}z^{1-\delta}.
\end{equation}
\end{prop}
We are ready to prove Theorem \ref{thm:exponent}.
\begin{proof}[Proof of Theorem \ref{thm:exponent}]
From exactly the same argument as of Section 5.3 in \cite{scc},
the difference of the counting functions and the corresponding
two main terms in \eqref{eq:cubic_density} are,
for any parameter $Q\leq X$ and $y\geq X^{3/5}$, bounded by
\begin{equation}\label{eq:difference-bound}
\ll\sum_{q<Q}E_q(r,y,X)+y^{1+\epsilon}+X/Q^{1-\epsilon},
\end{equation}
where
\begin{equation}
E_q(r,y,X)
=X^{3/8}\sum_{\mu_n<z_q}\frac{|c_{r,q}(\mu_n)|}{\mu_n^{5/8}}
+X^{3/8}\left(\frac{X^3}{y^4}\right)^{\rho/4}\sum_{\mu_n\geq z_q}\frac{|c_{r,q}(\mu_n)|}{\mu_n^{5/8+\rho/4}}.
\end{equation}
Here $\rho\geq3$ is a positive integer and
$z_q$ is another parameter which we can choose freely for each $q$.
By Proposition \ref{prop:Dirichlet_bound}, for $z_q\geq r^{-2}q^{-2}$,
\begin{equation}\label{eq:bound_E_q_z}
E_q(r,y,X)
\ll X^{3/8}r^{7/8+\epsilon}q^{7/8+\epsilon}
+X^{3/8}r^{1+\epsilon}q^{1+\epsilon}z_q^{3/8}
+X^{3/8}r^{1+\epsilon}q^{1+\epsilon}z_q^{3/8}
\left(\frac{X^3}{y^4z_q}\right)^{\rho/4}.
\end{equation}
For $q$ satisfying $X^3/y^4\geq r^{-2}q^{-2}$,
we choose $z_q=X^3/y^4$ and get the bound
\begin{equation}\label{eq:bound_E_q}
E_q(r,y,X)
\ll
X^{3/8}r^{7/8+\epsilon}q^{7/8+\epsilon}
+X^{3/2}r^{1+\epsilon}q^{1+\epsilon}y^{-3/2}.
\end{equation}
If $X^3/y^4\leq r^{-2}q^{-2}$, we choose $z_q=r^{-2}q^{-2}$.
Then $\frac{X^3}{y^4 z_q} \leq 1$, and so the latter two terms in the right hand side of \eqref{eq:bound_E_q_z}
are bounded by the first, so that \eqref{eq:bound_E_q} holds for such $q$ as well.
Hence \eqref{eq:difference-bound} is
\begin{equation}
\ll X^{3/8}r^{7/8+\epsilon}Q^{15/8+\epsilon}+X^{3/2}r^{1+\epsilon}Q^{2+\epsilon}y^{-3/2}+y^{1+\epsilon}+X/Q^{1-\epsilon}.
\end{equation}
Our theorem follows by choosing $y=X/Q$ and $Q=X^{5/23}r^{-7/23}$.
\end{proof}

\section{Remarks}\label{sec_remarks}
We give some remarks.
First, we counted $S_3$-sextic fields $\widetilde K$
with specifying the $3$-adic completion $A$ of $K$,
and by the same method we may specify
any finite number of local completions of $K$.
In particular for a fixed prime $p\neq3$, the ratio
of the contributions of $S_3$-sextic fields
whose splitting type of $p$ is
$(111111)$, $(222)$, $(33)$, $(1^21^21^2)$ and $(1^31^3)$
for the first and second main terms of \eqref{eq:strong}
are respectively given by
\[
\frac16:\frac12:\frac13:\frac1p:\frac1{p^{4/3}}
\quad
\text{and}
\quad
\frac{1+\frac{2}{p^{1/3}}+\frac{1}{p^{2/3}}}{6}
:\frac{1+\frac{1}{p^{2/3}}}{2}
:\frac{1-\frac{1}{p^{1/3}}+\frac{1}{p^{2/3}}}{3}
:\frac{1+\frac{1}{p^{1/3}}}{p}
:\frac{1}{p^{13/9}}.
\]

For $p = 3$ the last term should be replaced by
$3^{-5/3} + 2 \cdot 3^{-7/3}$ and $3^{-17/9} + 2 \cdot 3^{-22/9}$ respectively. For the splitting types
$(1^2 1^2 1^2)$ and $(1^3 1^3)$ there are often multiple possibilities for $K \otimes \Q_p$, depending on $p$,
and the terms above can be further subdivided following the tables
in Section 6.2 of \cite{scc}. Note that for any $p$ the sum of the first three
entries (corresponding to fields unramified at $p$) is $1$ and $1 + p^{-2/3}$ respectively.

All of this also follows from the methods of \cite{befo} or \cite{bhwo}. In our case the error term
remains the same, except that now it also depends (polynomially) on the prime(s) $p$.

In addition, by the same method, we can prove
the analogue of the power-saving remainder term
\eqref{eq:weak} for relative $S_3$-extensions over an arbitrary base number field $F$.
This would use the generalization of \eqref{eq:cubic_density}
over $F$, whose proof will appear elsewhere.
The exponent of $X$ in the $O$-term depends (only) on the degree $[F:\Q]$.

\section{Numerical experiments}\label{sec_data}
Finally, we compared our result \eqref{eq:strong} for $N^{\pm}_6(X; S_3)$ to numerical data.
Our data
weakly confirms \eqref{eq:strong}, but it suggests the presence of one or more additional secondary terms. Indeed, our data
will demonstrate several curious phenomena for which we don't have a satisfactory explanation.

We computed tables of $N_6^{\pm}(X; S_3)$ using two distinct methods:
\begin{itemize}
\item
We began with a direct approach, which allowed us to
tabulate $N_6^{\pm}(X; S_3)$ for $X \leq 5 \cdot 10^{18}$.
We used Belabas's \url{cubic} program
\cite{bel_cubic} to generate a list of all cubic fields $K$ with $|\Disc(K)| < (5/3)^{1/2} 10^9$, including generating 
polynomials. We have $\Disc(\widetilde{K}) = \Disc(K)^2 |\Disc(F)|$, where $F$ is the quadratic resolvent of $K$,
and as $|\Disc(F)| \geq 3$ we were able to tabulate $S_3$-fields with discriminant bounded by $5 \cdot 10^{18}$.

We used Lemma \ref{lem:relation_disc} to compute 
$\Disc(\widetilde{K})$ in terms of $\Disc(K)$. In particular, $\Disc(\widetilde{K})$ is determined by $\Disc(K)$ apart from the power of
2, which depends on whether or not $K$ is totally ramified at 2. For the power of 2,
Belabas's program 
outputs a binary cubic form $f = a u^3 + b u^2 v + c u v^2 + d v^3$ which corresponds to the maximal
order $\calO_K$, and 2 is totally ramified in $K$ if and only if $f$ has a triple root $(\textnormal{mod} \ 2)$, i.e., if
\begin{equation}
(a, b, c, d) \ (\textnormal{mod} \ 2) \in \{ (1, 1, 1, 1), (1, 0, 0, 0), (0, 0, 0, 1) \}.
\end{equation}
We used this condition to check the ramification at 2 and therefore to compile our list of $S_3$-sextic extensions.

This approach is somewhat inefficient: we also obtained many fields $\widetilde{K}$ with larger discriminant, 
which we had no choice but to discard.
\item
Recently, Cohen and the second author \cite{CT3} proved an explicit formula
enumerating cubic fields by their quadratic resolvents. As the referee (of the present
paper) suggested, the approach of \cite{CT3} is ideal for counting
$N_6^{\pm}(X; S_3)$, 
and so
we computed $N_6^{\pm}(X; S_3)$ for $X \leq 10^{23}$; we refer to \cite{CT3} for details.
\end{itemize}
Code implementing each of these algorithms, in Java and in PARI/GP \cite{pari} respectively, is available from the second
author's website; to reproduce our results using either program, the reader must also
download and run  \url{cubic}. (Belabas informs us that this functionality may be incorporated into a future version
of PARI/GP; this has the potential to make computations beyond $X = 10^{23}$ practical.)

As we will see, our data is a rather odd match to our theoretical
investigations, and the reader might be forgiven for speculating that our data is in error. To that end we note
that implementing redundant algorithms for $X \leq 5 \cdot 10^{18}$ 
allowed us to double check our results.

This brings us to the actual data, which we quote from \cite{CT3}. The tables below list $N_6^{\pm}(X; S_3)$ for
various $X$ between $10^{12}$ and $3 \cdot 10^{23}$. The columns labeled \eqref{eq:strong} give the values
predicted by \eqref{eq:strong}, which are consistently too high. (The bare main terms of Theorem \ref{thm:maintheorem}
are still higher.)
\\
\\
\begin{center}
\begin{tabular}{c | c | cccc}
$X$ & $N_6^+(X; S_3)$ & \eqref{eq:strong} & \eqref{eq:stronger} & Error \\ \hline
$10^{12}$ & 690 & 756 & 709 & .031\\
$10^{13}$ & 1650 & 1762 &1682 & .027\\ 
$10^{14}$& 3848 & 4045&3910 & .025\\
$10^{15}$& 8867 & 9181 &8955  & .021\\
$10^{16}$& 20062 & 20658 & 20276 & .021\\
$10^{17}$& 45054 & 46159 & 45513 & .021\\
$10^{18}$& 100335 & 102555 & 101460 & .022\\
$10^{19}$& 222939 & 226801 & 224943  & .020\\
$10^{20}$& 492335 & 499647 & 496490  & .020\\
$10^{21}$& 1083761 & 1097214 & 1091842  & .020\\
$10^{22}$& 2378358 & 2402995 & 2393842  & .019\\
$10^{23}$& 5207310 & 5250840 & 5235221  & .018\\
- & - & - & - & -\\

\end{tabular}
\ \
\begin{tabular}{c | c | ccc}
$X$ &$N_6^-(X; S_3)$ &  \eqref{eq:strong} & \eqref{eq:stronger} & Error \\ \hline
$10^{12}$ & 2809 & 2979 & 2828 & .079\\
$10^{13}$ & 6315 & 6613 & 6362 & .073\\ 
$10^{14}$& 14121 & 14617 & 14199 & .064\\
$10^{15}$& 31276 & 32192 & 31492 & .062 \\
$10^{16}$& 68972 & 70683 & 69507 & .061 \\
$10^{17}$& 151877 & 154800 & 152820 & .055\\
$10^{18}$& 333398 & 338279 & 334938 & .049 \\
$10^{19}$& 729572 & 737847 & 732195  & .044\\
$10^{20}$& 1592941 & 1606792 & 1597213  & .039\\
$10^{21}$& 3470007 & 3494240 & 3477974  & .036\\
$10^{22}$& 7550171 & 7589746 & 7562074  & .031\\
$10^{23}$& 16399890 & 16468453 & 16421298  & .028\\
$3 \cdot 10^{23}$& 23738460 & 23824734 & 23763890  & .026\\

\end{tabular}
\end{center}
To explain the apparent discrepancy between the data and \eqref{eq:strong},
we tried an amended heuristic. If $|\Disc(\widetilde{K})| < X$, then $K$ cannot
be totally ramified at any prime $> X^{1/4}$. This suggests multiplying the main term
by a factor
\begin{equation}\label{eqn_mod1}
\prod_{p > X^{1/4}} \frac{1 + p^{-1}}{1 + p^{-1} + p^{-4/3}}
\sim
1 - \sum_{p > X^{1/4}} p^{-4/3}
\sim 
1 - \int_{X^{1/4}}^{\infty} \frac{t^{-4/3}}{\log t} dt
\sim 1 - \frac{12 X^{-1/12}}{\log X}.
\end{equation}
(The approximations above are rather simple, so we verified numerically
that improving any of them leads only to minor
differences.)
Similarly, for the secondary term we incorporate the correction term
\begin{equation}\label{eqn_mod2}
\prod_{p > X^{1/4}} \frac{1 + p^{-2/3}+ p^{-1} + p^{-4/3}}
{1 + p^{-2/3}+ p^{-1} + p^{-4/3} + p^{-13/9}}
\sim 1 - \sum_{p > X^{1/4}} p^{-13/9} + O(p^{-19/9}) \sim 1 - 9 \frac{X^{-1/9}}{\log X}.
\end{equation}
This suggests the asymptotic formula
\begin{equation}\label{eq:stronger}
N_6^\pm(X;S_3)
\sim
\frac{C^\pm}{12}\prod_p c_p\cdot X^{1/3} \bigg(1 - \frac{12 X^{-1/12}}{ \log X} \bigg)
+\frac{4K^\pm\zeta(1/3)}{5\Gamma(2/3)^3}\prod_p k_p\cdot X^{5/18} \bigg(1 - 9 \frac{X^{-1/9}}{\log X} \bigg).
\end{equation}

With these corrections, we obtained the values listed under \eqref{eq:stronger} in our tables.
These values are more accurate, but still do not seem to closely match the data.

The final column labeled `Error' gives the relative error estimate $\frac{\eqref{eq:strong} - N_6^{\pm}(X, S_3)}{X^{5/18}}$. This column suggests that the secondary term in \eqref{eq:strong} is likely to be relevant,
but the evidence is not overwhelming. Our heuristics also do not explain why the relative error is larger for
negative discriminants, but (apparently) converges faster.

We tried other variations of our heuristics as well. As described earlier, we experimented with 
improving the estimates in \eqref{eqn_mod1} and \eqref{eqn_mod2}
(e.g. evaluating the integrals in \eqref{eqn_mod1} and \eqref{eqn_mod2} numerically instead
of using the approximation $\log t \sim \log X$ and evaluting them). This made only a very minor difference, and it adjusted
our counts upward rather than downward. Also, we observed that in fact no prime
larger than $\frac{X}{\sqrt[4]{3}}$ can totally ramify (as an $S_3$-sextic field has a nontrivial
quadratic resolvent), and we tried an accordingly modified version of \eqref{eqn_mod1} and \eqref{eqn_mod2}. 
These modified heuristics still produced data which were too high.

\subsection*{Arithmetic progressions}
Our work in \cite{scc} found and explained interesting discrepancies in the distribution of cubic field
discriminants in arithmetic progressions. For example, the following table lists 
the number 
of cubic fields $K$ with $0 < \Disc(K) < 2 \cdot 10^6$ and $\Disc(K) \equiv a \ (\textmod \ m)$ for $m = 5$ and $7$.
The ``predicted'' row is the sum of the $X$ and $X^{5/6}$ terms of the asymtptotic formula 
proved in \cite{scc}.

\begin{center}
\begin{tabular}{l | c | c | c | c | c  }
Discriminant modulo 5 & 0 & 1 & 2 & 3 & 4  \\ \hline
Actual count & 21277 & 22887 & 22751 & 22748 & 22781 \\
Predicted & 21307 & 22757 & 22757 & 22757 & 22757 \\
\end{tabular}
\end{center}

\begin{center}
\begin{tabular}{l | c | c | c | c | c | c | c}
Discriminant modulo 7 & 0 & 1 & 2 & 3 & 4 & 5 & 6 \\ \hline
Actual count & 15330 & 17229 & 14327 & 15323 & 17027 & 18058 & 15150 \\
Predicted & 15316 & 17209 & 14277 & 15316 & 17024 &  18063 & 15131 \\
\end{tabular}
\end{center}
The results $(\textmod \ 5)$
could have been predicted by Davenport and Heilbronn \cite{DH}.
In contrast, the $X^{5/6}$
term of the asymptotic is different for every residue class $a \ (\textmod \ 7)$.
We proved this in \cite{scc}; these results are explained by the existence of
nontrivial sextic characters $(\textmod \ 7)$, a phenomenon that could
have been predicted earlier by Datskovsky and Wright \cite{DW2}.
\\
\\
We briefly investigated analogous questions for $S_3$-sextic field discriminants, and
we quickly found interesting behavior which our methods could not explain.

For example, $S_3$-sextic field discriminants seem to not be equidistributed modulo $5$!
Using the algorithm of \cite{CT3}, we computed the following
data for $S_3$-sextic fields $\widetilde{K}$ of negative discriminant (where there are no cyclic cubic fields),
unramified at $2$ and $3$ (to eliminate wild ramification), and with $0 < - \Disc(\widetilde{K}) < X$:

\begin{center}
\begin{tabular}{l | c | c | c | c | c  }
$X$ & 0 & 1 & 2 & 3 & 4  \\ \hline
$10^{16}$ & 5034 & 3974 & 4091 & 4027 & 4075\\
$10^{17}$ & 11211 & 8817 & 8967 & 8833 & 9075\\
$10^{18}$ & 24816 & 19530 & 19872 & 19395 & 19902\\
$10^{19}$ & 54582 & 42917 & 43623 & 42972 & 43615\\
$10^{20}$ & 119354 & 94222 & 95303 & 94175 & 95428\\
$10^{21}$ & 261627 & 205997 & 208080 & 205916 & 208632\\
$10^{22}$ & 570179 & 449574 & 453456 & 449432 & 454119\\
$10^{23}$ & 1243107 & 980023 & 985513 & 978812 & 986670\\
$3 \cdot 10^{23}$ & 1801227 & 1420062 & 1427778 & 1418371 & 1429022\\

$(10^{20})$ & 122687 & 96553 & 96553 &96553 &96553 \\ 
$(3 \cdot 10^{23})$  & 1824995 & 1437452 &1437452 &1437452  & 1437452 \\

\end{tabular}
\end{center}
Each entry counts the number of $\widetilde{K}$ with $\Disc(\widetilde{K}) \pmod 5$; note that the discriminants are
negative.
The last two rows are predictions from \eqref{eq:strong}, modified as described in 
Section \ref{sec_remarks} for the primes $2$, $3$, and $5$. For $p = 5$ the $0$ column is
the contribution from fields ramified at $5$; the remainder is divided into four
equal parts, as predicted by our methods above and in \cite{scc}.\footnote{In particular, the 
secondary terms of counting functions for cubic field discriminants, twisted
by nontrivial Dirichlet characters $(\textmod \ 5)$, vanish; see Section 6.4 of \cite{scc}.}

The surplus of discriminants divisible by $5$ is predicted by
Lemma \ref{lem:relation_disc}: for any cubic field $K$ totally ramified at $5$, we know that
$\Disc(\widetilde{K}) \leq \frac{1}{25} \Disc(K)^3$, and so many such fields have small discriminant. 
However, we were surprised to observe a surplus
of field discriminants $\equiv 2, 4 \ (\textmod \ 5)$.  Certainly this is not predicted by any analysis involving
the Shintani zeta function. We looked for other heuristic explanations, for example using the fact that
\begin{equation}
\Disc(\widetilde{K}) \equiv \pm (p_1 p_2 \dots p_m)^{-1} \ (\textmod \ 5),
\end{equation}
where $p_1, p_2, \cdots, p_m > 5$ are the primes ramified but not 
totally ramified in $K$, but we did not find any convincing explanation.
\\
\\
In conclusion, \eqref{eq:strong} and probably also its generalization to arithmetic progressions, appear to be correct -- 
but our experiments have uncovered additional phenomena which call for explanation.
Naturally we hope to see further work on this topic in the future!

\end{document}